\documentclass[11pt,fleqn]{article}

\usepackage{amsmath,amssymb,amsthm}
\usepackage[left=3cm,right=3cm,top=2.45cm,bottom=2.45cm]{geometry}
\usepackage[colorlinks=true,urlcolor=blue,
citecolor=red,linkcolor=blue,linktocpage,pdfpagelabels,
bookmarksnumbered,bookmarksopen]{hyperref}
\usepackage[hyperpageref]{backref}

\newcommand{\abs}[1]{\left|#1\right|}
\newcommand{\bgset}[1]{\big\{#1\big\}}
\newcommand{\closure}[1]{\overline{#1}}
\newcommand{\dint}{\ds{\int}}
\newcommand{\ds}[1]{\displaystyle #1}
\newcommand{\dualp}[3][]{\left(#2,#3\right)_{#1}}
\newcommand{\hquad}{\hspace{0.08in}}
\newcommand{\incl}{\subset}
\newcommand{\intpart}[1]{\left[#1\right]}
\newcommand{\norm}[2][]{\left\|#2\right\|_{#1}}

\newcommand{\PS}[1]{$(\text{PS})_{#1}$}
\newcommand{\restr}[2]{#1|_{#2}}
\newcommand{\seq}[1]{\left(#1\right)}
\newcommand{\set}[1]{\left\{#1\right\}}

\newcommand{\F}{{\cal F}}
\newcommand{\G}{{\cal G}}
\newcommand{\M}{{\cal M}}
\newcommand{\N}{{\cal N}}
\newcommand{\R}{\mathbb R}
\newcommand{\RP}{\R \text{P}}
\newcommand{\Z}{\mathbb Z}

\def\ocirc#1{\ifmmode\setbox0=\hbox{$#1$}\dimen0=\ht0
\advance\dimen0 by1pt\rlap{\hbox to\wd0{\hss\raise\dimen0
\hbox{\hskip.2em$\scriptscriptstyle\circ$}\hss}}#1\else
{\accent"17 #1}\fi}

\DeclareMathOperator{\dvg}{div}

\newtheorem{lemma}{Lemma}[section]
\newtheorem{theorem}[lemma]{Theorem}

\numberwithin{equation}{section}

\title{\bf Asymptotic behavior of the eigenvalues\\ of the $p(x)$-Laplacian\thanks{{\rm MSC2010:} Primary 34L15, 34L20, Secondary 35J62,	74E10
\newline \smallskip \indent\; {\em Key Words and Phrases:} $p(x)$-Laplacian, eigenvalues, asymptotic behavior.
\newline The second author was supported by the 2009 MIUR project: ``Variational and Topological Methods in the Study of Nonlinear Phenomena''}}
\author{\bf Kanishka Perera\\
Department of Mathematical Sciences\\
Florida Institute of Technology\\
150 W University Blvd, Melbourne, FL 32901, USA\\
\em kperera@fit.edu\\
[\bigskipamount]
\bf Marco Squassina\\
Dipartimento di Informatica\\
Universit\`a degli Studi di Verona\\
C\'a Vignal 2, Strada Le Grazie 15, I-37134 Verona, Italy\\
\em marco.squassina@univr.it}
\date{}

\begin{document}

\maketitle

\begin{abstract}
We obtain asymptotic estimates for the eigenvalues of the $p(x)$-Laplacian defined consistently with a homogeneous notion of first eigenvalue recently introduced in the literature.
\end{abstract}

\section{Introduction}

Let $\Omega$ be a bounded domain in $\R^n,\, n \ge 1$ and let $p \in C(\closure{\Omega},(1,\infty))$. The purpose of this paper is to study the asymptotic behavior of the eigenvalues of the problem
\begin{equation} \label{1.1}
- \dvg \left(\abs{\frac{\nabla u}{K(u)}}^{p(x)-2} \frac{\nabla u}{K(u)}\right) = \lambda\, S(u)\, \abs{\frac{u}{k(u)}}^{p(x)-2} \frac{u}{k(u)}, \quad u \in W^{1,p(x)}_0(\Omega),
\end{equation}
where
\[
K(u) = \norm[p(x)]{\nabla u}, \qquad k(u) = \norm[p(x)]{u}, \qquad S(u) = \frac{\dint_\Omega \abs{\frac{\nabla u(x)}{K(u)}}^{p(x)} dx}{\dint_\Omega \abs{\frac{u(x)}{k(u)}}^{p(x)} dx}.
\]
The equation in \eqref{1.1} was derived by Franzina and Lindqvist in \cite{MR3040343} as the Euler-Lagrange equation
arising from minimizing the Rayleigh quotient $K(u)/k(u)$ over $W^{1,p(x)}_0(\Omega) \setminus \set{0}$.
It was shown there that the first eigenvalue $\lambda_1 > 0$ and has an associated eigenfunction $\varphi_1 > 0$.

We recall that the variable exponent Lebesgue space $L^{p(x)}(\Omega)$ consists of all measurable functions $u$ on $\Omega$ with the Luxemburg norm
\[
\norm[p(x)]{u} := \inf \set{\nu > 0 : \int_\Omega \abs{\frac{u(x)}{\nu}}^{p(x)} \frac{dx}{p(x)} \le 1} < \infty.
\]
The Sobolev space $W^{1,p(x)}(\Omega)$ consists of functions $u \in L^{p(x)}(\Omega)$
with a distributional gradient $\nabla u \in L^{p(x)}(\Omega)$, and the norm in this space is
$\norm[p(x)]{u} + \norm[p(x)]{\nabla u}$. The space $W^{1,p(x)}_0(\Omega)$ is the completion of $C^\infty_0(\Omega)$
with respect to the norm above, and has the equivalent norm $\norm[p(x)]{\nabla u}$. We refer the reader to Diening et al.\! \cite{MR2790542} for details on these spaces.

It was shown in \cite{MR3040343} that
\[
\dualp{K'(u)}{v} = \frac{\dint_\Omega \abs{\dfrac{\nabla u(x)}{K(u)}}^{p(x)-2} \dfrac{\nabla u(x)}{K(u)}\, \cdot \nabla v(x)\, dx}{\dint_\Omega \abs{\frac{\nabla u(x)}{K(u)}}^{p(x)} dx}, \quad u, v \in W^{1,p(x)}_0(\Omega)
\]
and
\[
\dualp{k'(u)}{v} = \frac{\dint_\Omega \abs{\dfrac{u(x)}{k(u)}}^{p(x)-2} \dfrac{u(x)}{k(u)}\ v(x)\, dx}{\dint_\Omega \abs{\frac{u(x)}{k(u)}}^{p(x)} dx}, \quad u, v \in W^{1,p(x)}_0(\Omega),
\]
so the eigenvalues and eigenfunctions of \eqref{1.1} on the manifold
$$
\M = \bgset{u \in W^{1,p(x)}_0(\Omega) : k(u) = 1}
$$
coincide with the critical values and critical points of $\widetilde{K} := \restr{K}{\M}$. In
the next section we will show that $\widetilde{K}$ satisfies the \PS{} condition, so we can
define an increasing and unbounded sequence of eigenvalues of \eqref{1.1} by a minimax scheme.
Although the standard scheme uses Krasnoselskii's genus, we prefer to use a cohomological index
as shown in \cite{MR1998432} by the first author since this gives additional Morse theoretic
information that is often useful in applications.

Let us recall the definition of the $\Z_2$-cohomological index of Fadell and Rabinowitz \cite{MR57:17677}. Let $\F$ denote the class of symmetric subsets of $\M$. For $M \in \F$, let $\overline{M} = M/\Z_2$ be the quotient space of $M$ with each $u$ and $-u$ identified, let $f : \overline{M} \to \RP^\infty$ be the classifying map of $\overline{M}$, and let $f^\ast : H^\ast(\RP^\infty) \to H^\ast(\overline{M})$ be the induced homomorphism of the Alexander-Spanier cohomology rings. Then the cohomological index of $M$ is defined by
\[
i(M) = \begin{cases}
\sup\, \bgset{m \ge 1 : f^\ast(\omega^{m-1}) \ne 0}, & M \ne \emptyset\\[5pt]
0, & M = \emptyset,
\end{cases}
\]
where $\omega \in H^1(\RP^\infty)$ is the generator of the polynomial ring $H^\ast(\RP^\infty) = \Z_2[\omega]$. For example, the classifying map of the unit sphere $S^{m-1}$ in $\R^m,\, m \ge 1$ is the inclusion $\RP^{m-1} \incl \RP^\infty$, which induces isomorphisms on $H^q$ for $q \le m-1$, so $i(S^{m-1}) = m$.

Set
\begin{equation} \label{1.45}
\lambda_j := \inf_{\substack{M \in \F\\[1pt]
i(M) \ge j}}\, \sup_{u \in M}\, \widetilde{K}(u), \quad j \ge 1.
\end{equation}
Then $\seq{\lambda_j}$ is a sequence of eigenvalues of \eqref{1.1} and $\lambda_j \nearrow \infty$. Moreover,
\[
\lambda_j < \lambda \le \lambda_{j+1} \implies i(\widetilde{K}^\lambda) = j,
\]
where $\widetilde{K}^\lambda = \bgset{u \in \M : \widetilde{K}(u) < \lambda}$, so
\begin{equation} \label{1.46}
i(\widetilde{K}^\lambda) = \# \bgset{j : \lambda_j < \lambda} \qquad \forall \lambda \in \R
\end{equation}
(see Propositions 3.52 and 3.53 of Perera et al.\! \cite{MR2640827}). Our main result is the following.

\begin{theorem} \label{mainth}
If $1 < p^- \le p(x) \le p^+ < \infty$ for all $x \in \Omega$ and
\[
\sigma := n \left(\frac{1}{p^-} - \frac{1}{p^+}\right) < 1, \qquad \tau := \left(\frac{1}{p^-} - \frac{1}{p^+}\right) |\Omega| < 1,
\]
then there are constants $C_1, C_2 > 0$ depending only on $n$ and $p^\pm$ such that
\[
C_1\, |\Omega|\, (\lambda/\kappa)^{n/(1 + \sigma)} \le \# \bgset{j : \lambda_j < \lambda} \le C_2\, |\Omega|\, (\kappa\, \lambda)^{n/(1 - \sigma)} \qquad \text{for $\lambda > 0$ large},
\]
where $|\Omega|$ is the Lebesgue measure of $\Omega$ and $\kappa = (1 + \tau)^{1/p^-}/(1 - \tau)^{1/p^+}$.
\end{theorem}

This result is a contribution towards understanding the spectrum of the $p(x)$-Laplacian, which many researchers have recently found to be somewhat puzzling. For example, it is currently unknown if the first eigenvalue is simple, or if a given positive eigenfunction is automatically a first eigenfunction. Affirmative answers were given to both of these questions for the usual eigenvalue problem for the $p$-Laplacian,
\begin{equation} \label{1.41}
- \dvg \left(|\nabla u|^{p-2}\, \nabla u\right) = \lambda\, |u|^{p-2}\, u, \quad u \in W^{1,p}_0(\Omega),
\end{equation}
where $p > 1$ is a constant, in Lindqvist \cite{MR90h:35088,MR1139483} (see also \cite{MR2462954}). It should be noted that, in the case when $p$ is constant, \eqref{1.1} reduces, not to the problem \eqref{1.41}, which is homogeneous of degree $p-1$, but rather to the nonlocal problem
\[
- \dvg \left(\frac{|\nabla u|^{p-2}\, \nabla u}{\norm[p]{\nabla u}^{p-1}}\right) = \lambda\, \frac{|u|^{p-2}\, u}{\norm[p]{u}^{p-1}}, \quad u \in W^{1,p}_0(\Omega)
\]
that has been normalized to be homogeneous of degree $0$. The estimate
\[
C_1\, |\Omega|\, \lambda^n \le \# \bgset{j : \lambda_j < \lambda} \le C_2\, |\Omega|\, \lambda^n \qquad \text{for $\lambda > 0$ large}
\]
that Theorem \ref{mainth} gives for the eigenvalues of this problem should be compared with the estimate
\[
C_1\, |\Omega|\, \lambda^{n/p} \le \# \bgset{j : \lambda_j < \lambda} \le C_2\, |\Omega|\, \lambda^{n/p} \qquad \text{for $\lambda > 0$ large}
\]
obtained by Friedlander in \cite{MR1017063} for \eqref{1.41} (see also Garc{\'{\i}}a Azorero and Peral Alonso \cite{MR954263}). Caliari and Squassina \cite{MR3122341} have recently developed a numerical method to compute the first eigenpair of the problem \eqref{1.1} and investigate the symmetry breaking phenomena with respect to the constant case.

In the course of proving Theorem \ref{mainth}, we will also establish the same asymptotic estimates for the eigenvalues of the problem
\begin{equation} \label{1.5}
- \dvg \left(\abs{\frac{\nabla u}{L(u)}}^{p(x)-2} \frac{\nabla u}{L(u)}\right) = \mu\, T(u)\, \abs{\frac{u}{l(u)}}^{p(x)-2} \frac{u}{l(u)}, \quad u \in W^{1,p(x)}(\Omega),
\end{equation}
where
\[
L(u) = \norm[p(x)]{\nabla u}, \qquad l(u) = \norm[p(x)]{u}, \qquad T(u) = \frac{\dint_\Omega \abs{\frac{\nabla u(x)}{L(u)}}^{p(x)} dx}{\dint_\Omega \abs{\frac{u(x)}{l(u)}}^{p(x)} dx}
\]
(which coincide with $K$, $k$, and $S$, respectively, on $W^{1,p(x)}_0(\Omega)$). The eigenvalues and eigenfunctions of this problem on
\[
\N = \bgset{u \in W^{1,p(x)}(\Omega) : l(u) = 1}
\]
coincide with the critical values and critical points of $\widetilde{L} := \restr{L}{\N}$. Let $\G$ denote the class of symmetric subsets of $\N$ and set
\[
\mu_j := \inf_{\substack{N \in \G\\[1pt]
i(N) \ge j}}\, \sup_{u \in N}\, \widetilde{L}(u), \quad j \ge 1.
\]
Then $\seq{\mu_j}$ is a sequence of eigenvalues of \eqref{1.5}, $\mu_j \nearrow \infty$, and
\[
i(\widetilde{L}^\mu) = \# \bgset{j : \mu_j < \mu} \qquad \forall \mu \in \R,
\]
where $\widetilde{L}^\mu = \bgset{u \in \N : \widetilde{L}(u) < \mu}$. Since $W^{1,p(x)}(\Omega) \supset W^{1,p(x)}_0(\Omega)$ and $\restr{l}{W^{1,p(x)}_0(\Omega)} = k$, we have $\N \supset \M$, and $\restr{\widetilde{L}}{\M} = \widetilde{K}$, so $\mu_j \le \lambda_j$ for all $j$. We will see that, under the hypotheses of Theorem \ref{mainth},
\[
C_1\, |\Omega|\, (\mu/\kappa)^{n/(1 + \sigma)} \le \# \bgset{j : \mu_j < \mu} \le C_2\, |\Omega|\, (\kappa\, \mu)^{n/(1 - \sigma)} \qquad \text{for $\mu > 0$ large}.
\]

Finally, for the sake of completeness, let us also mention that a different notion of first eigenvalue for the $p(x)$-Laplacian, that does not make use of the Luxemburg norm, has been considered in the past literature, namely,
\[
\lambda^\ast_1 = \inf_{u \in W^{1,p(x)}_0(\Omega) \setminus \set{0}}\, \frac{\dint_\Omega |\nabla u|^{p(x)}\, dx}{\dint_\Omega |u|^{p(x)}\, dx}.
\]
In this framework, $\lambda \in \R$ and $u \in W^{1,p(x)}_0(\Omega) \setminus \set{0}$ are an eigenvalue and an eigenfunction of the $p(x)$-Laplacian, respectively, if
\[
\int_\Omega |\nabla u|^{p(x)-2}\, \nabla u \cdot \nabla v\, dx = \lambda \int_\Omega |u|^{p(x)-2}\, uv\, dx \qquad \forall v \in W^{1,p(x)}_0(\Omega)
\]
(this should be compared with \eqref{1.1}). Let $\Lambda$ denote the set of all eigenvalues of this problem. If the function $p(x)$ is a constant $p > 1$, then it is well-known that this problem admits an increasing sequence of eigenvalues, $\sup \Lambda = + \infty$, and $\inf \Lambda = \lambda_{1,p} > 0$, the first eigenvalue of the $p$-Laplacian (see Lindqvist \cite{MR90h:35088,MR1139483,MR2462954}). For general $p(x)$, $\Lambda$ is a nonempty infinite set, $\sup \Lambda = + \infty$, and $\inf \Lambda = \lambda^\ast_1$ (see Fan et al.\! \cite{MR2107835}). In contrast to the situation when minimizing the Rayleigh quotient with respect to the Luxemburg norm, one often has $\lambda^\ast_1 = 0$, and $\lambda^\ast_1 > 0$ only under special conditions. In \cite{MR2107835}, the authors provide sufficient conditions for $\lambda^\ast_1$ to be zero or positive. In particular, if $p(x)$ has a strict local minimum (or maximum) in $\Omega$, then $\lambda^\ast_1 = 0$. If $n > 1$ and there is a vector $\ell \ne 0$ in $\R^n$ such that for every $x \in \Omega$, the map $t \mapsto p(x + t \ell)$ is monotone on $\bgset{t : x + t \ell \in \Omega}$, then $\lambda^\ast_1 > 0$. Finally, if $n = 1$, then $\lambda^\ast_1 > 0$ if and only if the function $p(x)$ is monotone.

\section{Compactness}

In this section we will show that $\widetilde{K}$ satisfies the \PS{} condition. Here and in the next section we will make use of the well-known Young's inequality
\begin{equation} \label{1.15}
ab \le \Big(1 - \frac{1}{p}\Big) a^{p/(p-1)} + \frac{1}{p}\, b^p \qquad \forall a, b \ge 0,\, p > 1.
\end{equation}

\begin{lemma} \label{Lemma 1.1}
For $u \ne 0$ in $L^{p(x)}(\Omega)$ and all $v \in L^{p(x)}(\Omega)$,
\begin{equation} \label{1.2}
\abs{\dualp{k'(u)}{v}} \le \norm[p(x)]{v}.
\end{equation}
\end{lemma}

\begin{proof}
Equality holds in \eqref{1.2} if $v = 0$, so suppose $v \ne 0$. We have
\begin{equation} \label{1.3}
\abs{\dualp{k'(u)}{v}} \le \frac{\dint_\Omega \abs{\frac{u(x)}{k(u)}}^{p(x)-1} |v(x)|\, dx}{\dint_\Omega \abs{\frac{u(x)}{k(u)}}^{p(x)} dx}.
\end{equation}
Taking $a = |u(x)/k(u)|^{p(x)-1},\, b = |v(x)/k(v)|,\, p = p(x)$ in \eqref{1.15} and integrating over $\Omega$ gives
\[
\int_\Omega \abs{\frac{u(x)}{k(u)}}^{p(x)-1} \abs{\frac{v(x)}{k(v)}} dx \le \int_\Omega \abs{\frac{u(x)}{k(u)}}^{p(x)} dx - \int_\Omega \abs{\frac{u(x)}{k(u)}}^{p(x)} \frac{dx}{p(x)} + \int_\Omega \abs{\frac{v(x)}{k(v)}}^{p(x)} \frac{dx}{p(x)}.
\]
The last two integrals are both equal to $1$, so this shows that the right-hand side of \eqref{1.3} is less than or equal to $k(v) = \norm[p(x)]{v}$.
\end{proof}

\begin{lemma} \label{Lemma 1.2}
$K'$ is a mapping of type $(S_+)$, i.e., if $u_j \rightharpoonup u$ in $W^{1,p(x)}_0(\Omega)$ and
\[
\varlimsup_{j \to \infty} \dualp{K'(u_j)}{u_j - u} \le 0,
\]
then $u_j \to u$ in $W^{1,p(x)}_0(\Omega)$.
\end{lemma}

\begin{proof}
Since
\[
\dualp{K'(u_j)}{u_j} = K(u_j) = \norm[p(x)]{\nabla u_j}
\]
and
\[
\dualp{K'(u_j)}{u} = \dualp{k'(\nabla u_j)}{\nabla u} \le \norm[p(x)]{\nabla u}
\]
by Lemma \ref{Lemma 1.1},
\[
\varlimsup_{j \to \infty}\, \norm[p(x)]{\nabla u_j} \le \varlimsup_{j \to \infty} \dualp{K'(u_j)}{u_j - u} + \norm[p(x)]{\nabla u} \le \norm[p(x)]{\nabla u} \le \varliminf_{j \to \infty}\, \norm[p(x)]{\nabla u_j},
\]
so that $\norm[p(x)]{\nabla u_j} \to \norm[p(x)]{\nabla u}$. The conclusion follows since $W^{1,p(x)}_0(\Omega)$ is uniformly convex.
\end{proof}

\begin{lemma}
For all $c \in \R$, $\widetilde{K}$ satisfies the {\em \PS{c}} condition, i.e., every sequence $\seq{u_j} \subset \M$ such that $\widetilde{K}(u_j) \to c$ and $\widetilde{K}'(u_j) \to 0$ has a convergent subsequence.
\end{lemma}

\begin{proof}
We have
\begin{equation} \label{1.4}
K(u_j) \to c, \qquad K'(u_j) - c_j\, k'(u_j) \to 0
\end{equation}
for some sequence $\seq{c_j} \subset \R$. Since $\dualp{K'(u_j)}{u_j} = K(u_j)$ and $\dualp{k'(u_j)}{u_j} = k(u_j) = 1$, $c_j \to c$. Since $\seq{u_j}$ is bounded in $W^{1,p(x)}_0(\Omega)$, for a renamed subsequence and some $u \in W^{1,p(x)}_0(\Omega)$, $u_j \rightharpoonup u$ in $W^{1,p(x)}_0(\Omega)$ and $u_j \to u$ in $L^{p(x)}(\Omega)$. By Lemma \ref{Lemma 1.1},
\[
\abs{\dualp{k'(u_j)}{u_j - u}} \le \norm[p(x)]{u_j - u} \to 0,
\]
so the second limit in \eqref{1.4} now gives $\dualp{K'(u_j)}{u_j - u} \to 0$ as $j \to \infty$. Then we conclude that $u_j \to u$ strongly in $W^{1,p(x)}_0(\Omega)$, in light of Lemma \ref{Lemma 1.2}.
\end{proof}

\section{Proof of Theorem \ref{mainth}}

Let $\sigma$, $\tau$, and $\kappa$ be as in Theorem \ref{mainth}.

\begin{lemma} \label{Lemma 1.7}
We have
\begin{equation} \label{1.7}
\frac{\norm[p^-]{u}}{(1 + \tau)^{1/p^-}} \le \norm[p(x)]{u} \le \frac{\norm[p^+]{u}}{(1 - \tau)^{1/p^+}} \qquad \forall u \in L^{p^+}(\Omega),
\end{equation}
and hence
\[
\frac{1}{\kappa}\, \frac{\norm[p^-]{\nabla u}}{\norm[p^+]{u}} \le \frac{\norm[p(x)]{\nabla u}}{\norm[p(x)]{u}} \le \kappa\, \frac{\norm[p^+]{\nabla u}}{\norm[p^-]{u}} \qquad \forall u \in W^{1,p^+}(\Omega) \setminus \set{0}.
\]
\end{lemma}

\begin{proof}
Equality holds throughout \eqref{1.7} if $u = 0$, so suppose $u \ne 0$. Taking $a = 1,\, b = |u(x)/\norm[p(x)]{u}|^{p^-},\, p = p(x)/p^-$ in \eqref{1.15}, dividing by $p^-$, and integrating over $\Omega$ gives
\[
\frac{1}{\norm[p(x)]{u}^{p^-}} \int_\Omega |u(x)|^{p^-}\, \frac{dx}{p^-} \le \int_\Omega \left(\frac{1}{p^-} - \frac{1}{p(x)}\right) dx + \int_\Omega \bigg|\frac{u(x)}{\norm[p(x)]{u}}\bigg|^{p(x)} \frac{dx}{p(x)}.
\]
The first integral is equal to $\norm[p^-]{u}^{p^-}$ and the last integral is equal to $1$, so this gives the first inequality in \eqref{1.7}. Now taking $a = 1,\, b = |u(x)/\norm[p(x)]{u}|^{p(x)},\, p = p^+/p(x)$ in \eqref{1.15}, dividing by $p(x)$, and integrating over $\Omega$ gives
\[
\int_\Omega \bigg|\frac{u(x)}{\norm[p(x)]{u}}\bigg|^{p(x)} \frac{dx}{p(x)} \le \int_\Omega \left(\frac{1}{p(x)} - \frac{1}{p^+}\right) dx + \frac{1}{\norm[p(x)]{u}^{p^+}} \int_\Omega |u(x)|^{p^+}\, \frac{dx}{p^+}.
\]
The first integral is equal to $1$ and the last integral is equal to $\norm[p^+]{u}^{p^+}$, so this gives the second inequality in \eqref{1.7}.
\end{proof}

Recall that the genus and the cogenus of $M \in \F$ are defined by
\[
\gamma(M) = \inf\, \bgset{m \ge 1 : \exists \text{ an odd continuous map } g : M \to S^{m-1}}
\]
and
\[
\widetilde{\gamma}(M) = \sup\, \bgset{\widetilde{m} \ge 1 : \exists \text{ an odd continuous map } \widetilde{g} : S^{\widetilde{m}-1} \to M},
\]
respectively. If there are odd continuous maps $S^{\widetilde{m} - 1} \to M \to S^{m - 1}$, then $\widetilde{m} \le i(M) \le m$ by the monotonicity of the index, so $\widetilde{\gamma}(M) \le i(M) \le \gamma(M)$. Since $\widetilde{K}^\lambda \subset \widetilde{L}^\lambda$, this gives
\begin{equation} \label{1.8}
\widetilde{\gamma}(\widetilde{K}^\lambda) \le i(\widetilde{K}^\lambda) \le i(\widetilde{L}^\lambda) \le \gamma(\widetilde{L}^\lambda) \qquad \forall \lambda \in \R.
\end{equation}
Set
\[
\widehat{K}(u) := \norm[p^+]{\nabla u}, \quad u \in \widehat{\M} := \bgset{u \in W^{1,p^+}_0(\Omega) : \norm[p^-]{u} = 1}
\]
and
\[
\widehat{L}(u) := \norm[p^-]{\nabla u}, \quad u \in \widehat{\N} := \bgset{u \in W^{1,p^+}(\Omega) : \norm[p^+]{u} = 1},
\]
and let $\widehat{K}^\lambda = \bgset{u \in \widehat{\M} : \widehat{K}(u) < \lambda}$ and $\widehat{L}^\mu = \bgset{u \in \widehat{\N} : \widehat{L}(u) < \mu}$.

\begin{lemma} \label{Lemma 1.8}
We have
\[
\widetilde{\gamma}(\widehat{K}^{\lambda/\kappa}) \le \widetilde{\gamma}(\widetilde{K}^\lambda), \qquad \gamma(\widetilde{L}^\lambda) \le \gamma(\widehat{L}^{\kappa\, \lambda}) \qquad \forall \lambda \in \R.
\]
\end{lemma}

\begin{proof}
Lemma \ref{Lemma 1.7} gives the odd continuous maps
\[
\widehat{K}^{\lambda/\kappa} \to \widetilde{K}^\lambda, \hquad u \mapsto \frac{u}{\norm[p(x)]{u}}, \qquad \widetilde{L}^\lambda \cap W^{1,p^+}(\Omega) \to \widehat{L}^{\kappa\, \lambda}, \hquad u \mapsto \frac{u}{\norm[p^+]{u}},
\]
and the inclusion $\widetilde{L}^\lambda \cap W^{1,p^+}(\Omega) \incl \widetilde{L}^\lambda$ is a homotopy equivalence by Palais \cite[Theorem 17]{MR0189028} since $W^{1,p^+}(\Omega)$ is a dense linear subspace of $W^{1,p(x)}(\Omega)$, so the conclusion follows.
\end{proof}

\begin{lemma} \label{Lemma 1.6}
Let $0 < \delta < 1$, consider the homothety $\Omega \to \delta\, \Omega,\, x \mapsto \delta x =: y$, and write $u(x) = v(y)$. Then
\[
\frac{\norm[p^+]{\nabla v}}{\norm[p^-]{v}} = \delta^{- \sigma - 1}\, \frac{\norm[p^+]{\nabla u}}{\norm[p^-]{u}}, \qquad \frac{\norm[p^-]{\nabla v}}{\norm[p^+]{v}} = \delta^{\sigma - 1}\, \frac{\norm[p^-]{\nabla u}}{\norm[p^+]{u}} \qquad \forall u \in W^{1,p^+}(\Omega) \setminus \set{0}.
\]
\end{lemma}

\begin{proof}
Straightforward.
\end{proof}

\begin{lemma} \label{Lemma 1.5}
If $\Omega_1$ and $\Omega_2$ are disjoint subdomains of $\Omega$ such that $\closure{\Omega}_1 \cup \closure{\Omega}_2 = \closure{\Omega}$, then
\[
\widetilde{\gamma}(\widehat{K}_{\Omega_1}^\lambda) + \widetilde{\gamma}(\widehat{K}_{\Omega_2}^\lambda) \le \widetilde{\gamma}(\widehat{K}_\Omega^\lambda), \qquad \gamma(\widehat{L}_\Omega^\lambda) \le \gamma(\widehat{L}_{\Omega_1}^{\lambda'}) + \gamma(\widehat{L}_{\Omega_2}^{\lambda'}) \qquad \forall \lambda < \lambda',
\]
where the subscripts indicate the corresponding domains.
\end{lemma}

\begin{proof}
Since $\widehat{K}_\Omega^\lambda$ contains $\widehat{K}_{\Omega_1}^\lambda$ and $\widehat{K}_{\Omega_2}^\lambda$, if $\widetilde{\gamma}(\widehat{K}_{\Omega_1}^\lambda)$ or $\widetilde{\gamma}(\widehat{K}_{\Omega_2}^\lambda)$ is infinite, then so is $\widetilde{\gamma}(\widehat{K}_\Omega^\lambda)$ and hence the first inequality holds. So let $\widetilde{m}_i := \widetilde{\gamma}(\widehat{K}_{\Omega_i}^\lambda) < \infty$ and let $\widetilde{g}_i : S^{\widetilde{m}_i-1} \to \widehat{K}_{\Omega_i}^\lambda$ be an odd continuous map for $i = 1, 2$. Write $y \in S^{\widetilde{m}_1+\widetilde{m}_2-1}$ as $y = (y_1,y_2) \in \R^{\widetilde{m}_1} \oplus \R^{\widetilde{m}_2}$, set $\abs{y_2} = t$, and let
\[
\widetilde{g}(y) = \begin{cases}
\widetilde{g}_1(y_1), & t = 0\\[7.5pt]
\dfrac{(1 - t)\, \widetilde{g}_1(y_1/\sqrt{1 - t^2}) + t\, \widetilde{g}_2(y_2/t)}{\norm[p^-]{(1 - t)\, \widetilde{g}_1(y_1/\sqrt{1 - t^2}) + t\, \widetilde{g}_2(y_2/t)}}, & 0 < t < 1\\[17pt]
\widetilde{g}_2(y_2), & t = 1.
\end{cases}
\]
Clearly, $\widetilde{g}(y) \in \widehat{K}_\Omega^\lambda$ for $t = 0, 1$. For $0 < t < 1$,
\[
\widehat{K}_\Omega(\widetilde{g}(y)) < \lambda\, \frac{\big[(1 - t)^{p^+} + t^{p^+}\big]^{1/p^+}}{\big[(1 - t)^{p^-} + t^{p^-}\big]^{1/p^-}} \le \lambda
\]
since $p \mapsto [(1 - t)^p + t^p]^{1/p}$
on $(1,\infty)$ is nonincreasing. So $\widetilde{g} : S^{\widetilde{m}_1+\widetilde{m}_2-1} \to \widehat{K}_\Omega^\lambda$ is an odd continuous map and hence $\widetilde{\gamma}(\widehat{K}_\Omega^\lambda) \ge \widetilde{m}_1 + \widetilde{m}_2$.

Since the second inequality holds if $\gamma(\widehat{L}_{\Omega_1}^{\lambda'})$ or $\gamma(\widehat{L}_{\Omega_2}^{\lambda'})$ is infinite, let $m_i := \gamma(\widehat{L}_{\Omega_i}^{\lambda'}) < \infty$ and let $g_i : \widehat{L}_{\Omega_i}^{\lambda'} \to S^{m_i-1}$ be an odd continuous map for $i = 1, 2$. For $u \in \widehat{L}_\Omega^\lambda$, let $u_i = \restr{u}{\Omega_i}$, $\rho_i = \norm[p^+]{u_i}$, and $\widetilde{u}_i = u_i/\rho_i$ if $\rho_i \ne 0$. Fix $\lambda'' \in (\lambda,\lambda')$ such that $(\lambda/\lambda'')^{p^+} \ge 1/2$, take smooth cutoff functions $\eta,\, \zeta : [0,\infty) \to [0,1]$ such that $\eta = 0$ near zero, $\eta = 1$ on $[[1 - (\lambda/\lambda'')^{p^+}]^{1/p^+},\infty)$, $\zeta = 1$ on $[0,\lambda'']$ and $\zeta = 0$ on $[\lambda',\infty)$, and let
\begin{equation} \label{3.3}
g(u) = \dfrac{\big(\eta(\rho_1)\, \zeta(\widehat{L}_{\Omega_1}(\widetilde{u}_1))\, g_1(\widetilde{u}_1),\eta(\rho_2)\, \zeta(\widehat{L}_{\Omega_2}(\widetilde{u}_2))\, g_2(\widetilde{u}_2)\big)}{\sqrt{\eta(\rho_1)^2\, \zeta(\widehat{L}_{\Omega_1}(\widetilde{u}_1))^2 + \eta(\rho_2)^2\, \zeta(\widehat{L}_{\Omega_2}(\widetilde{u}_2))^2}},
\end{equation}
with the understanding that $\eta(\rho_i)\, \zeta(\widehat{L}_{\Omega_i}(\widetilde{u}_i))\, g_i(\widetilde{u}_i) = 0$ if $\rho_i = 0$. We claim that the denominator is greater than or equal to $1$. The claim is clearly true if $u_1 = 0$ or $u_2 = 0$, so suppose $u_1 \ne 0$ and $u_2 \ne 0$. Since $\rho_1^{p^+} + \rho_2^{p^+} = \norm[p^+]{u}^{p^+} = 1$, either $\rho_1 \ge 1/2^{1/p^+}$ or $\rho_2 \ge 1/2^{1/p^+}$, and since $1/2^{1/p^+} \ge [1 - (\lambda/\lambda'')^{p^+}]^{1/p^+}$, then either $\eta(\rho_1) = 1$ or $\eta(\rho_2) = 1$. Moreover, if $\widehat{L}_{\Omega_i}(\widetilde{u}_i) \ge \lambda$ for $i = 1, 2$, then
\[
1 = \rho_1^{p^+} + \rho_2^{p^+} \le \rho_1^{p^-} + \rho_2^{p^-} \le \frac{\norm[p^-]{\nabla u_1}^{p^-} + \norm[p^-]{\nabla u_2}^{p^-}}{\lambda^{p^-}} = \frac{\norm[p^-]{\nabla u}^{p^-}}{\lambda^{p^-}} < 1,
\]
a contradiction, so either $\zeta(\widehat{L}_{\Omega_1}(\widetilde{u}_1)) = 1$ or $\zeta(\widehat{L}_{\Omega_2}(\widetilde{u}_2)) = 1$. Consequently, we are done if $\eta(\rho_1) = 1$ and $\eta(\rho_2) = 1$, so assume that one of them, say $\eta(\rho_2)$, is less than $1$. Then $\eta(\rho_1) = 1$. Moreover, $\rho_2 < [1 - (\lambda/\lambda'')^{p^+}]^{1/p^+}$ and hence
\[
\widehat{L}_{\Omega_1}(\widetilde{u}_1) = \frac{\norm[p^-]{\nabla u_1}}{\rho_1} \le \frac{\norm[p^-]{\nabla u}}{(1 - \rho_2^{p^+})^{1/p^+}} < \lambda'',
\]
so $\zeta(\widehat{L}_{\Omega_1}(\widetilde{u}_1)) = 1$. Thus, the denominator in \eqref{3.3} is greater than or equal to $\eta(\rho_1)\, \zeta(\widehat{L}_{\Omega_1}(\widetilde{u}_1)) = 1$. So $g : \widehat{L}_\Omega^\lambda \to S^{m_1+m_2-1}$ is an odd continuous map and hence $\gamma(\widehat{L}_\Omega^\lambda) \le m_1 + m_2$.
\end{proof}

We are now ready to prove Theorem \ref{mainth}.

\begin{proof}[Proof of Theorem \ref{mainth}]
Continuously extend $p$ to the whole space, with the same bounds $p^-$ and $p^+$, using the Tietze extension theorem. Let $Q$ be the unit cube in $\R^n$, fix $\lambda_0 > \max \bgset{\inf \widehat{K}_Q, \inf \widehat{L}_Q}$, and set
\[
r = \widetilde{\gamma}(\widehat{K}^{\lambda_0}_Q), \qquad s = \gamma(\widehat{L}^{\lambda_0}_Q).
\]
Then for $\lambda' > \lambda > \lambda_0$ and any two cubes $Q_{a_\lambda}$ and $Q_{b_{\lambda'}}$ of sides $a_\lambda = (\lambda_0/\lambda)^{1/(1 + \sigma)}$ and $b_{\lambda'} = (\lambda_0/\lambda')^{1/(1 - \sigma)}$, respectively, Lemma \ref{Lemma 1.6} gives the odd homeomorphisms
\[
\widehat{K}^{\lambda_0}_Q \to \widehat{K}^\lambda_{Q_{a_\lambda}}, \hquad u \mapsto \frac{v}{\norm[p^-]{v}}, \qquad \widehat{L}^{\lambda_0}_Q \to \widehat{L}^{\lambda'}_{Q_{b_{\lambda'}}}, \hquad u \mapsto \frac{v}{\norm[p^+]{v}},
\]
so
\[
\widetilde{\gamma}(\widehat{K}^\lambda_{Q_{a_\lambda}}) = r, \qquad \gamma(\widehat{L}^{\lambda'}_{Q_{b_{\lambda'}}}) = s.
\]
Now it follows from Lemma \ref{Lemma 1.5} that if $Q_a$ is a cube of side $a > 0$, then
\[
r \intpart{\frac{a}{a_\lambda}}^n \le \widetilde{\gamma}(\widehat{K}^\lambda_{Q_a}), \qquad \gamma(\widehat{L}^\lambda_{Q_a}) \le s \left(\intpart{\frac{a}{b_{\lambda'}}} + 1\right)^n,
\]
where $\intpart{\cdot}$ denotes the integer part. Thus, there are constants $C_1, C_2 > 0$, independent of $a$, $\lambda$, and $\lambda'$, such that
\begin{equation} \label{1.9}
C_1\, a^n \lambda^{n/(1 + \sigma)} \le \widetilde{\gamma}(\widehat{K}^\lambda_{Q_a}), \qquad \gamma(\widehat{L}^\lambda_{Q_a}) \le C_2\, a^n (\lambda')^{n/(1 - \sigma)}, \qquad \lambda < \lambda' \text{ large}.
\end{equation}

Let $\varepsilon > 0$ and let $\Omega_\varepsilon, \Omega^\varepsilon$ be unions of cubes with pairwise disjoint interiors such that $\Omega_\varepsilon \subset \Omega \subset \Omega^\varepsilon$ and $|\Omega^\varepsilon \setminus \Omega_\varepsilon| < \varepsilon$. Then
\[
C_1\, |\Omega_\varepsilon|\, \lambda^{n/(1 + \sigma)} \le \widetilde{\gamma}(\widehat{K}^\lambda_{\Omega_\varepsilon}) \le \widetilde{\gamma}(\widehat{K}^\lambda), \qquad \gamma(\widehat{L}^\lambda) \le \gamma(\widehat{L}^\lambda_{\Omega^\varepsilon}) \le C_2\, |\Omega^\varepsilon|\, (\lambda')^{n/(1 - \sigma)}
\]
by \eqref{1.9} and Lemma \ref{Lemma 1.5}. Letting $\varepsilon \searrow 0$, $\lambda' \searrow \lambda$, and combining with \eqref{1.46}, \eqref{1.8}, and Lemma \ref{Lemma 1.8} yields the conclusion.
\end{proof}

\def\cdprime{$''$}

\end{document}